\documentclass[12pt,a4paper]{lms}

\usepackage{amsfonts}
\usepackage{amssymb}
\usepackage[pdftex,colorlinks=true,
                      pdfstartview=FitV,
                      linkcolor=blue,
                      citecolor=blue,
                      urlcolor=blue,
          ]{hyperref}

\newtheorem{thm}{Theorem}[section] 
\newtheorem{lemma}[thm]{Lemma}     
\newtheorem{cor}[thm]{Corollary}
\newtheorem{prp}[thm]{Proposition}

\newnumbered{assertion}{Assertion}    
\newnumbered{conjecture}{Conjecture}  
\newnumbered{df}{Definition}
\newnumbered{hypothesis}{Hypothesis}
\newnumbered{rem}{Remark}
\newnumbered{note}{Note}
\newnumbered{observation}{Observation}
\newnumbered{problem}{Problem}
\newnumbered{question}{Question}
\newnumbered{algorithm}{Algorithm}
\newnumbered{ex}{Example}
\newunnumbered{notation}{Notation} 

\newcommand{\tr}{{\rm tr}}
\newcommand{\mod}{{\rm mod}}

\title[Non-linear Group Actions]
{Non-linear Group Actions with Polynomial Invariant Rings and a
Structure Theorem for Modular Galois Extensions} 

\author{Peter Fleischmann and Chris Woodcock}

\classno{13A50, 13B05, 20C20 (primary), 13B05,13B40,20C05 (secondary).
 Please refer to {\tt http://www.ams.org/msc/} for a list of codes}

\begin{document}
\maketitle

\begin{abstract}
Let $G$ be a finite $p$-group and $k$ a field of characteristic $p>0$. We show that
$G$ has a \emph{non-linear} faithful action on a polynomial ring $U$ of dimension $n=\mathrm{log}_p(|G|)$ such that the invariant ring $U^G$ is also polynomial.
This contrasts with the case of \emph{linear and graded} group actions with polynomial rings of invariants,
where the classical theorem of
Chevalley-Shephard-Todd and Serre requires $G$ to be generated by pseudo-reflections.
\\
Our result is part of a general theory of ``trace surjective $G$-algebras", which, in
the case of $p$-groups, coincide with the Galois ring-extensions in the sense of \cite{chr}.
We consider the \emph{dehomogenized symmetric algebra} $D_k$, a polynomial ring with non-linear $G$-action, containing $U$ as a retract and we show that $D_k^G$ is a polynomial ring.
Thus $U$ turns out to be \emph{universal} in the sense that
every trace surjective $G$-algebra can be constructed from $U$ by ``forming
quotients and extending invariants". As a consequence we obtain a general structure theorem for Galois-extensions with given $p$-group as Galois group and any prescribed
commutative $k$-algebra $R$ as invariant ring. This is a generalization of the Artin-Schreier-Witt theory of modular Galois field extensions of degree $p^s$.
\end{abstract}


\section{Introduction} \label{Intro}

Let $k$ be a field and let $G$ be a finite group acting \emph{linearly} on a polynomial ring
$A:=k[X_1,\cdots,X_n]$. Due to a classical theorem of Chevalley-Shephard-Todd and Serre it is known
that if the invariant ring $A^G$ is a polynomial ring, then $G$ is generated by pseudo-reflections.
The converse is also true if $|G|\in k^*$, but is known to fail if ${\rm char}(k)\ |\ |G|$.
The assumption on the group action to be \emph{linear} in degree one and degree preserving on $A$
is essential for that result. In contrast to this we prove that if ${\rm char}(k)=p>0$, then
\emph{every finite $p$-group} $G$ has a \emph{non-linear} faithful action on a polynomial ring $U$ of Krull dimension $n=\mathrm{log}_p(|G|)$, such that the ring of invariants $U^G$ is also a polynomial ring. This is part of a more general theory of group actions on
noetherian, not necessarily artinian, $k$-algebras $A$ with surjective trace function
$$\tr:\ A\to A^G,\ a\mapsto \sum_{g\in G} ag.$$
This condition is equivalent to $A$ being a projective $kG$-module. In this paper
we begin the systematic study (mainly in the commutative and non-artinian case) of
these \emph{trace-surjective} $G$\emph{-algebras}. They have a
beautiful structure theory and they arise naturally in various contexts,
such as modular invariant theory, linear algebraic groups, Galois theory, ramification theory of
commutative rings and in the cohomology of finite groups.
\\
In invariant theory for example, one usually considers the symmetric algebra
$A = \mathrm{Sym}_k(W)$, where $W$ is a faithful finitely-generated
$kG$-module. Then $G$ acts on the graded $k$-algebra $A$ and the ring of
invariants $A^G$ is finitely generated. If the characteristic $\mathrm{char}(k)$ of $k$
is coprime to $|G|$, then $A^G$ can be generated by invariants of
degree at most $|G|$ (see \cite{N1} for $\mathrm{char}(k) = 0$ and
\cite{pf_adv} or \cite{Fog} for $\mathrm{char}(k) = p > 0$). If $p$ divides $|G|$, however,
there is no such bound that depends only
on the size of $G$ (see e.g.\cite{Derksen:Kemper} and \cite{symonds_bound}).
However, there exists a non-zero homogeneous transfer element $c = {\rm tr}(f)\in A^G$
with $f\in A$ of degree less than $|G|$ (see \cite{homloc}).
Let $A_c:=A[\frac{1}{c}]$, the localization of $A$ at $c$. Then it is
shown in \cite{homloc} that the invariant ring $A_c^G$ always
satisfies a ``monomial degree bound'' close to the Noether bound. In fact,
$A_c$ is a trace-surjective $G$-algebra and, if $G$ is a $p$-group
and $k$ has characteristic $p$, then this implies that $A_c$ is a free
$A_c^G[G]$-module of rank $1$ with basis $\frac{f}{c}$ (see Theorem \ref{first_main}).\\
Localizations like these do arise naturally in the theory of linear algebraic groups:
Assume for the moment that $k$ is algebraically closed and that $G\le{\rm SL}_n(k)$ is
a finite subgroup. The left multiplication action of
$G$ on ${\rm Mat}_n(k)$ induces a homogeneous right regular action on the coordinate ring
$k[M]:=k[{\rm Mat}_n(k)]\cong k[X_{ij}\ |\ 1\le i,j\le n]$ with ${\rm det}:={\rm det}(X_{ij})\in k[M]^G$.
It can be shown that ${\rm det}\in\sqrt{\tr(k[M])}$, i.e. $\tr(f)=({\rm det})^N$ for some $N\in \mathbb{N}$
and some $f\in k[M]$.
It follows that the coordinate ring $k[{\rm GL}_n]=k[M][1/{\rm det}]$ is a
trace-surjective $G$-algebra. Since epimorphic images of trace-surjective algebras are
again trace-surjective (see Theorem \ref{first_main} (iii) in the case of $p$-groups),
a similar conclusion holds, if
${\rm GL}_n$ is replaced by an arbitrary closed linear algebraic subgroup containing $G$ (see
Corollary \ref{alg_grps}, where this is proved in a different way). As a consequence we obtain a plethora of
examples of trace-surjective algebras, arising in the theory of algebraic groups.
We are indebted to Stephen Donkin and Bram Broer for bringing these to our attention.
\\
Another one of these localizations, which is particularly important for the theory to come,
is the ring of Laurent-polynomials $D_k[t,1/t]$, where $D_k$ is the \emph{dehomogenized symmetric algebra}
in the sense of Bruns-Herzog (\cite{BH} pg.38). This means that
$D_k = \mathrm{Sym}_k(kG)/(\alpha)$, with $\alpha=-1+\sum_{g\in G} X_g$
where $kG=\oplus_{g\in G}kX_g$.
\\
Now let $A$ be a commutative ring and let $G\le {\rm Aut}(A)$ be a finite subgroup with
ring of invariants $R:=A^G$. In the classical paper \cite{chr}, a notion of
Galois extensions of commutative rings is defined and it is shown
how the main results of the Galois theory of fields, in particular the correspondence
between subgroups $H\le G$ and intermediate Galois-extensions $A^H\le A$,
can be generalized to commutative rings, provided the $G$-action
satisfies some natural axioms. These axioms hold for example in the
``unramified step" $(A_{\rm q})^I\le A_{\rm Q}$ of a general ring extension,
where ${\rm Q}\in {\rm Spec}(A)$ with inertia group
$I:=\{g\in G\ |\ ag-a\in {\rm Q}\}\unlhd {\rm Stab}_G({\rm Q})$
and ${\rm q}={\rm Q}\cap A^I$.
We will show in Section \ref{ts_alg}, that for $p$-groups and $k$-algebras of
characteristic $p$, the concepts of Galois-extensions and trace-surjective
algebras are equivalent (see Corollary \ref{G_p_gr_ts_iff_gal}).
\\
In the cohomology of $p$-groups, a significant appearance of trace-surjective algebras has
recently been observed by one of the authors (\cite{cfw3}), motivating some
of the investigations in this paper.
\\
The following are the main results of this paper:

\begin{thm}\label{arb_p_grp_intro}
Let $G$ be an arbitrary finite $p$-group of order $p^n$ and $k$ a field of characteristic
$p>0$. Then the following hold:
\begin{enumerate}
\item There is a trace-surjective $G$-subalgebra $U\le D_k$,
such that $U$ is a polynomial ring of Krull-dimension $n$ and a retract of
$D_k$, i.e. $D_k=U\oplus I$ with a $G$-stable \emph{ideal} $I\unlhd D_k$ (see
Theorem \ref{arb_p_grp}).
\item The algebra $U$ can be constructed in such a way that $U=k[Y_0,\cdots,Y_{n-1}]$ with ``triangular"
$G$-action of the form $(Y_i)g=Y_i+f_i(Y_{i+1},\cdots,Y_{n-1})$ for $i=0,\cdots,n-1$ (see
Remark \ref{UG_is_poly_rem} (iii)).
\item The ring of invariants $U^G$ is also a polynomial ring of Krull-dimension $n$.
\end{enumerate}
\end{thm}
Let $\mathfrak{Ts}$ denote the category of all commutative trace-surjective
$G$-algebras, with morphisms being $G$-equivariant homomorphisms of $k$-algebras.
The algebra $U$ turns out to be a projective
\footnote{with respect to surjective functions rather than
epimorphisms} object in $\mathfrak{Ts}$, and its
Krull-dimension $\mathrm{log}_p(|G|)$ is minimal among
all finitely generated projective objects, at least if $k = \mathbb{F}_p$.
Finitely generated projective objects in $\mathfrak{Ts}$ are precisely the retracts
of the tensor powers of $D_k$. Since $D_k$ and its tensor powers are
polynomial algebras it is natural to ask whether or not every
finitely-generated projective algebra is a polynomial algebra
(see \cite{costa}, \cite{shpilrain} for related open questions/conjectures).
It would be interesting to determine all the projective objects in $\mathfrak{Ts}$
of minimal dimension. We will address these and related
categorical questions about $\mathfrak{Ts}$ in a separate paper.
\\
Let us now fix a polynomial retract $k[Y_0,Y_1,\cdots,Y_{n-1}]\cong U\le D_k$ with
polynomial ring of invariants $U^G=k[\sigma_0(\underline Y),\cdots,\sigma_{n-1}(\underline Y)]$ as described in Theorem \ref{arb_p_grp_intro} (see Theorem \ref{UG_is_poly} for more details).
Then we obtain the following

\begin{thm}[Structure theorem for $\mathfrak{Ts}$]\label{strct_thm_intro}
Let $|G|=p^n$, $k$ a field of characteristic $p>0$ and $R$ be a commutative $k$-algebra.
Then every algebra $A\in \mathfrak{Ts}$, with given ring of invariants $A^G=R$ is of the form
$$A\cong R[Y_0,\cdots,Y_{n-1}]/(\sigma_0(\underline Y)-r_0,\cdots,\sigma_{n-1}(\underline Y)-r_{n-1})$$
with suitable $r_0,\cdots,r_{n-1}\in R$, and $G$-action derived from the action on $U$.\\
The $\sigma_i$ can be determined explicitly and are of the form
$$\sigma_i(\underline Y)=Y_i-Y_i^p+\gamma_i(Y_{i+1},\cdots,Y_{n-1}).$$
\end{thm}
Since $A\in\mathfrak{Ts}$ if and only if $A^G\le A$ is a Galois-extension in the sense of
\cite{chr} (see Corollary \ref{G_p_gr_ts_iff_gal}), this can also be viewed as a structure theorem
for arbitrary $p$-group-Galois-extensions of commutative $k$-algebras.
A further application of Theorem \ref{arb_p_grp_intro} provides

\begin{thm}\label{D_k_G is poly_intro}
The invariant ring $D_k^G$ is a polynomial ring of Krull-dimension $|G|-1$.
\end{thm}

This result and Theorem \ref{arb_p_grp_intro} show that, for every finite $p$-group $G$,
there are faithful non-homogeneous modular representations as automorphisms of
polynomial rings of dimensions $|G|-1$ and $\mathrm{log}_p(|G|)$, respectively,
such that the corresponding rings of invariants are polynomial rings. As mentioned before,
this is quite different from the situation of faithful linear and \emph{graded} group actions on polynomial rings, where it is known that the rings of invariants being polynomial again requires
$G$ to be generated by (pseudo-) reflections. The paper is organized as follows:
\\[1mm]
In Section \ref{basic} we recall some basic results from the representation
theory of finite groups. Most of the material is well known
and can be found in standard textbooks, formulated and proved
for \emph{finitely generated} artinian rings and modules. Since
we would like to avoid this restriction, we will, where needed,
include short proofs of some of these results, in a more general
context.
\\
In Section \ref{ex_cp} we will look at the example where $G$ is
cyclic of order $p$. Lemma \ref{cycl_p_lemma} and Proposition
\ref{cycl_p_univ_model} provide an induction base for
the proof of Theorem \ref{arb_p_grp}, and Corollary
\ref{D_k_poly_for_Cp} for Theorem \ref{D_k_G is poly_intro}. We also state a
structure theorem for ``Artin-Schreier" extensions of
commutative rings, generalizing the corresponding
theorem for $C_p$-extensions of fields of characteristic $p>0$. This follows as a
special case from Theorem \ref{strct_thm_intro}, so
we omit a proof in Section \ref{ex_cp}.
\\
In Section \ref{ts_alg} we introduce and study the trace-surjective $G$-algebras
for arbitrary finite $p$-groups and particularly consider the \emph{dehomogenized ring}
$D_k$ and its quotients.
\\
In Section \ref{std_sub} we define ``standard subalgebras" of $D_k$, from which,
quite surprisingly, all objects in $\mathfrak{Ts}$ can be constructed by
forming a quotient, followed by an ``extension of invariants".
In particular we prove that there is always a polynomial ring $U$
of dimension $\mathrm{log}_p(|G|)$, which is a standard subalgebra and
we show that this is the minimal possible dimension of such a polynomial
ring, at least if $k=\mathbb{F}_p$.
\\
Section \ref{poly_strct} will finish the proofs of the main
theorems, using the results of Section \ref{std_sub}.
\\
Finally in Section \ref{concl_rem} we outline some future paths of
research and point out some open questions and conjectures
related to this work.
\\
\textsl{Acknowledgements}:
The authors would like to thank an anonymous referee for many helpful comments and
suggestions, which we hope considerably improve the readability of the paper.

\section{Basic observations}\label{basic}

{\bf Notation}:\ From now on throughout the paper $k$ will be a field of characteristic $p>0$ and $G$ will
be a finite $p$-group. All group actions on sets will be written as \emph{right} actions and
if $\Omega$ is a set on which the group $G$ acts, then
$\Omega^G:=\{\omega\in \Omega\ |\ \omega g=\omega\ \forall g\in G\}$, the set of $G$-fixed points
in $\Omega$. Usually $A$ will denote a commutative $k$-algebra, on which
$G$ acts by $k$-algebra automorphisms. As a shorthand we will refer to such an algebra $A$
as a $G$-algebra. For any ring $A$, its Krull-dimension will be denoted by ${\rm Dim}(A)$.
We will now collect some well known elementary facts needed later.

\begin{lemma} \label{observation1}
Let $W$ be a (right) $kG$ - module with finite dimensional submodules $V_i$ for $i\in I$.
Then
$$\sum_{i\in I} V_i = \oplus_{i\in I} V_i
\iff \sum_{i\in I} V_i^G = \oplus_{i\in I} V_i^G.$$
\end{lemma}

\begin{proof} We only need to prove ``$\Leftarrow$" and for this it suffices to show
$\sum_{\ell=1}^k V_{i_\ell} = \oplus_{\ell=1}^k V_{i_\ell}$ for any \emph{finite} number
$k\in \mathbb{N}$. The proof is by induction starting with $k=1$, where the claim
is obvious. Let $Y:= V_{i_{k+1}}\cap (\oplus_{\ell=1}^k
V_{i_\ell})$ with $i_{k+1}\not\in \{i_1,\cdots,i_k\}$, then
$V_{i_{k+1}}^G \cap Y = $
$$V_{i_{k+1}}^G \cap (\oplus_{\ell=1}^k V_{i_\ell}) =
V_{i_{k+1}}^G\cap (\oplus_{\ell=1}^k V_{i_\ell}^G)
= 0.$$ Since $V_{i_{k+1}}^G \le V_{i_{k+1}}$ is an essential
extension, i.e. $V_{i_{k+1}}^G\cap M\ne 0$ for every $0\ne M\le V_{i_{k+1}}$,
we conclude that $Y=0$.
\end{proof}

Now let $V$ be a (right) $kG$ - module. The
\emph{transfer map} is defined to be the $k$-linear map
$$\tr:=\tr^{(G)}:\ V\to V^G,\ v \mapsto \sum_{g\in G}\ vg.$$
If $W$ is another right $kG$-module, then $G$ has a natural action on
${\rm Hom}_k(V,W)$ by conjugation, i.e.
$\alpha^g(v) = (\alpha(vg^{-1}))g$ with ${\rm Hom}_k(V,W)^G = {\rm Hom}_{kG}(V,W).$
Note that $\tr(\alpha) = \sum_{x\in G} \alpha^x \in {\rm Hom}_{kG}(V,W)$.

\begin{prp}\label{higman} (Higman's projectivity criterion)
Let $V$ be a $kG$-module, then the following are equivalent:
\begin{enumerate}
\item There is $\alpha\in {\rm End}_k(V)$ with
$\tr(\alpha)={\rm id}_V$.
\item $V$ is a direct summand of $kG\otimes_k V.$
\item $V$ is a free $kG$-module.
\end{enumerate}
\end{prp}
\begin{proof} See e.g. \cite{Huppert:Blackburn} Theorem 7.8 pg.86, or
\cite{bens1:b} Proposition 3.6.4 pg.70.
\end{proof}



The following lemma tells us how to recognize free summands in
$kG$-modules, using the transfer map $\tr$. In the literature it
is usually stated for finitely generated $kG$-modules, but since
we need it without that hypothesis, we include a short proof:

\begin{lemma} \label{rec_free}
Let $V$ be an arbitrary $kG$-module. Then the following are
equivalent:
\begin{enumerate}
\item
${\rm tr}(V) \ne 0$;
\item there is a free direct summand $0\ne F\le V$ containing ${\rm tr}(V)$.
\end{enumerate}
Moreover $V$ is free if and only if ${\rm tr}(V) = V^G$. If $v\in V$
satisfies ${\rm tr}(v)\ne 0$, then
$\langle vg\ |g\in G\rangle \in kG-mod$ is free of rank one.
\end{lemma}
\begin{proof}
(i) $\Rightarrow$ (ii):\ Let $\{ w_i\ |\ i\in I\}$ be a $k$ - basis
of ${\rm tr}(V)$ with $0\ne w_i:={\rm tr}(v_i)$ for $v_i\in V$. The
homomorphism $\theta_i:\ kG\to V,\ f\mapsto v_if$
maps the one-dimensional socle $(kG)^G= k\cdot {\rm tr}(1_G)$
onto the line $k\cdot w_i\subseteq V^G$. Hence every $\theta_i$ is a
monomorphism. Let $V_i:= \theta_i(kG)$, then $(V_i)^G \cong
k\cdot w_i$, hence, by Lemma \ref{observation1}, $F:= \sum_{i\in I} V_i = \oplus_{i\in I} V_i$ is a free submodule
of $V$. It is well known that for $kG$-modules
the notions of finitely generated projective, injective and free
modules coincide. Since $kG$ is Noetherian, it follows from a result by
H. Bass (see \cite{BH} Remark 3.1.4 or \cite{rotman} Theorem 4.10), that arbitrary direct sums of
injective modules are injective, so $F$ is an injective
submodule of $V$ with $V\cong F\oplus W\ {\rm and}\  {\rm tr}(V) = {\rm tr}(F).$
\\
(ii) $\Rightarrow$ (i): Since $0\ne F \cong \oplus_{i\in
I}(kG)^{(i)}$ for some index set $I\ne\emptyset$, we obtain
${\rm tr}(V)\ge {\rm tr}(F) \ne 0.$\\
We have seen that every $kG$-module $V$ splits as $V\cong F\oplus
W$, such that $F$ is free with
$${\rm tr}(V)={\rm tr}(F) = F^G \le F^G \oplus W^G = V^G.$$
In this equation we have equality if and only if $W^G =0$, which is
equivalent to $W=0$ or to $V=F$ being a free $kG$-module.\\
If ${\rm tr}(v)\ne 0$ for $v\in V$, then $\langle vg\ |\ g\in G\rangle$ is of dimension
$\le |G|$ and has a free
summand; hence it must be free of rank one.
\end{proof}

\begin{lemma}\label{tr}
Let $A$ be a $G$-algebra, then the following are equivalent:
\begin{enumerate}
\item $1=\tr(a)$ for some $a\in A$.
\item $A^G=\tr(A)$.
\item $A$ is free as a $kG$-module.
\item There is a $kG$ -submodule $W\le A$, isomorphic
to the regular $kG$-module $V_{reg}$, such that $1_A\in W$.
\end{enumerate}
\end{lemma}
\begin{proof} Clearly (i) $\iff$ (ii), since $\tr(A)$ is a two-sided ideal
of $A^G$. \\
For $a\in A$ let $\mu_a\in {\rm End}_k(A)$ denote the
homomorphism given by left-multiplication, i.e. $\mu_a(s)=a\cdot s$
for all $s\in A$. Then
$$(\mu_a)^g(s) = (\mu_a(sg^{-1}))g = (a\cdot sg^{-1})g=(ag)\cdot s =
\mu_{ag}(s),$$ hence the map $\mu:\ A\to {\rm End}_k(A),\ a\mapsto \mu_a$ is a unitary homomorphism of $G$-algebras.
On the other hand the map $e:\ {\rm End}_k(A) \to A,\ \alpha\mapsto
\alpha(1)$ satisfies
$e(\alpha^g) = (\alpha(1g^{-1}))g =\alpha(1)g=(e(\alpha))g,$ hence it is a homomorphism of
$kG$-modules with $e({\rm id}_A)=1$. We have
$e\circ \mu = {\rm id}_A$ and $(\mu\circ e)({\rm id}_A) = {\rm id}_A$.
If $1=\tr(a)$, then ${\rm id}_A=\mu(1)=\mu(\tr(a)) = \tr(\mu(a))$. On the other
hand if ${\rm id}_A = \tr(\alpha)$, then $1= e({\rm id}_A) =
\tr(e(\alpha))$. It now follows from Proposition \ref{higman} that (i) and (iii)
are equivalent.\\
``(i) $\Rightarrow$ (iv)" follows from Lemma \ref{rec_free};\\
``(iv) $\Rightarrow$ (i)":\ Since $W\cong V_{reg}$,
$W^G=\tr(W)=k1_A.$
\end{proof}

\begin{df}\label{ts-alg_df}
Let $A$ be a $G$-algebra such that ${\rm tr}(A)=A^G$, or any of
the other conditions in Lemma \ref{tr} is satisfied.
Then $A$ will be called a {\bf trace-surjective} $G$-algebra
(or a ${\rm ts}$-algebra, if $G$ is clear from the context).
An element $a\in A$ with ${\rm tr}(a)=1$ will be called a
{\bf point} of $A$.\\
Recall that $\mathfrak{Ts}$ denotes the category of all commutative trace-surjective $G$-algebras,
with morphisms being $G$-equivariant homomorphisms of $k$-algebras.
\end{df}

\section{Example of cyclic group of order $p$}\label{ex_cp}

In this section let $G:=\langle g\rangle \cong C_p$ be cyclic of order $p$,
$V_{reg}:=\oplus_{g\in G} kX_g$ the regular $kG$-module with $X_{gh}:=(X_g)h$
for $g,h\in G$. We define $D_k(G):={\rm Sym}(V_{reg})/(\alpha)$ with $\alpha=-1+\sum_{g\in G}X_g.$
Then $D_k(G)=k[x_0,x_1,\cdots,x_{p-1}]$ with
$x_i:=(X_1)g^i+(\alpha)$ and $x_0+\cdots+x_{p-1}=1$ and $G$-action defined by
$x_i g=x_{i+1 \mod(p)}$.
The results in Lemma \ref{cycl_p_lemma} and Proposition \ref{cycl_p_univ_model}
will be used in section \ref{std_sub} as a base of
induction for the general case of arbitrary finite $p$-groups.
In the same way Corollary \ref{D_k_poly_for_Cp} will be used for Theorem \ref{D_k_G is poly_intro}.
Define $y:=\sum_{i=1}^{p-1} ix_i$, then (with $x_p:=x_0$):
$$y g=\sum_{i=1}^{p-1} ix_{i+1}=\sum_{i=1}^{p-1} (i+1)x_{i+1}-\sum_{i=1}^{p-1} x_{i+1}=
\sum_{i=1}^{p-1} ix_i-\sum_{i=1}^p x_i= y-1.$$
Hence we find inside $D_k(G)$ the polynomial subalgebra $U:=k[y]$ of Krull-dimension one, on which $G$ has the natural (inhomogeneous) action defined by $g:\ y\mapsto y-1$.

\begin{lemma}\label{cycl_p_lemma}
Let $k[T]$ be the polynomial ring with $G$-action by algebra homomorphisms, defined by
$T g:=T-1$. Let ${\rm tr}:\ k[T]\to k[T]^G$ be the transfer map $f\mapsto \sum_{g \in G} fg.$
Then
\begin{enumerate}
\item For $0\le j\le 2p-2$: ${\rm tr}(T^j)=\sum_{i=0}^{p-1}(T-i)^j=
\cases{-1&{\rm if}\ j=p-1\ {\rm or}\ j=2p-2\cr
1&{\rm if}\ j=p=2\cr
0&{\rm otherwise}}.$
\\[1mm]
\item For $0\le j\le p-1$:
$\sum_{i=1}^{p-1}i(T-i)^j=
\cases{-T&{\rm if}\ $j=p-1\ge 2$\cr
(T+1)^j&{\rm if}\ p=2\cr
1&{\rm if}\ j=p-2\cr
0&{\rm otherwise}}.$
\end{enumerate}
In particular ${\rm tr}(T^{p-1})=-1$.
\end{lemma}
\begin{proof}
The claims are obviously true for $p=2$, so we can assume $p>2$.
Set $\sigma:=T^p-T\in k[T]^G$. Assume first that $0\le j\le p-1$.
We have
$\sum_{i=0}^{p-1}(T-i)^j=$ $\sum_{i=0}^{p-1}\sum_{s=0}^j {j \choose s}(-i)^sT^{j-s}=$
$\sum_{s=0}^j{j \choose s}(-1)^s T^{j-s}\cdot \sum_{i=0}^{p-1}i^s=$
$-{j \choose p-1}T^{j-p+1},$
since $\sum_{i=0}^{p-1}i^s=
\cases{
\ 0 &$\forall\  0\le s<p-1$\cr
-1&{\rm for}\ $s=p-1$.}$
If $p\le j=p-1+k\le 2p-2$, then ${\rm tr}(T^j)=$ ${\rm tr}(T^pT^{k-1})=$
${\rm tr}(T^k)+\sigma {\rm tr}(T^{k-1})=$
${\rm tr}(T^k)$, and the result (i) follows.\\
We also have:
$\sum_{i=0}^{p-1}i(T-i)^j=\sum_{i=0}^{p-1}i\cdot \sum_{s=0}^j
{j\choose s}(-i)^sT^{j-s}=$
$$\sum_{s=0}^j (-1)^s{j\choose s}T^{j-s}[\sum_{i=0}^{p-1}i^{s+1}]=
(-1)^{p-2}{j\choose p-2}(-1)T^{j-(p-2)},$$
since for $0\le s \le p-1$ we have
$\sum_{i=0}^{p-1}i^{s+1}=-1\cdot \delta_{s,p-2}.$
The results in (ii) follow directly.
\end{proof}

\begin{prp}\label{cycl_p_univ_model}
The univariate polynomial ring $U:=k[y]\le D_k$ is a retract with point\\
$w=-y^{p-1}$\footnote{a ``reflexive point" as in
Definition \ref{reflexive_points}, if $p>2$}. In particular
there is a $G$-equivariant morphism of $k$-algebras
$$\theta=\theta^2:\ D_k\to D_k\ {\rm with}\ {\rm Im}(\theta)=k[y],$$
defined by the map $x_i\mapsto -(y-i)^{p-1}$, if $p>2$
and $\theta={\rm id}_{D_k}$, if $p=2$ such that
$$D_k=D_k^G \otimes_{k[\sigma]}k[y]\cong k[y]\oplus {\rm ker}(\theta),$$
with $k[y]^G=k[\sigma],$ a univariate polynomial ring in
$\sigma:=y^p-y.$
\end{prp}
\begin{proof} If $p=2$, $y=x_1$, $k[y]=D_k$ and the statements are trivially true,
so we can assume $p>2$. It follows from Lemma \ref{cycl_p_lemma} that
${\rm tr}(-y^{p-1})=1$, hence the map $\theta$ is a well-defined
$G$-equi\-vari\-ant morphism of $k$-algebras.
Again by Lemma \ref{cycl_p_lemma},
$\theta(y)=-\sum_{i=1}^{p-1}i(y-i)^{p-1}=y,$ so ${\rm Im}(\theta)=k[y]$ and
$\theta^2=\theta$, hence $D_k=k[y]\oplus {\rm ker}(\theta)$ (for the general case see Lemma \ref{endo_krit_standard}).
Let $\mathbb{L}$ be the quotient field of $D_k$, then clearly
$\mathbb{L}=\mathbb{L}^G[y]=\oplus_{i=0}^{p-1}\mathbb{L}^G y^i.$
Write $\ell\in\mathbb{L}$ in the form
$\ell=\sum_{i=0}^{p-1}\ell_iy^i$, then the formulae
in Lemma \ref{cycl_p_lemma} (i) show that $\ell_0=-{\rm tr}(\ell y^{p-1})+{\rm tr}(\ell)$
and $\ell_i=-{\rm tr}(\ell y^{p-1-i})$ for $1\le i\le p-1$. Picking $\ell\in D_k$ we see
$D_k=\oplus_{i=0}^{p-1} D_k^G y^i$.
Obviously $k[y]=\oplus_{i=0}^{p-1} k[\sigma] y^i$. Let
$f\in k[y]^G$, then $f=f\cdot {\rm tr}(-y^{p-1})=$ ${\rm tr}(-y^{p-1}f).$
Write $-y^{p-1}f=\sum_{i=0}^{p-1} f_iy^i$ with $f_i\in k[\sigma]$, then the formulae
in Lemma \ref{cycl_p_lemma} (i) again show $f=-f_{p-1}$, so $k[y]^G=k[\sigma]$.
Finally
$D_k^G \otimes_{k[\sigma]}k[y]\cong$
$D_k^G \otimes_{k[\sigma]}\oplus_{i=0}^{p-1} k[\sigma] y^i\cong$
$\oplus_{i=0}^{p-1} D_k^G y^i=D_k$.
\end{proof}

\begin{cor}\label{D_k_poly_for_Cp}
$D_k^G(G)$ is isomorphic to a polynomial ring in $p-1$ variables.
\end{cor}
\begin{proof}
If $p=2$, then $D_k=k[y]$ and therefore $D_k^G=k[sigma]$, so we can assume that $p>2$.
Set $\sigma:=y^p-y$, then
$D_k=D_k^G\otimes_{k[\sigma]} k[y]\cong \oplus_{i=0}^{p-1} D_k^Gy^i$.
Write $x_0=\sum_{i=0}^{p-1} b_iy^i$ with $b_i\in D_k^G$ and set
$C:=k[\sigma,b_i\ |\ i=0,\cdots,p-1]\le D_k^G.$
Since $yg=y-1$, $x_0 g\in C[y]$ and
$D_k=k[x_0{g^i} |\ g^i\in G]\le C[y]=\oplus_{i=0}^{p-1} C y^i\le\oplus_{i=0}^{p-1} D_k^G y^i=
D_k,$
hence $D_k^G=C.$ Moreover
$1={\rm tr}(x_0)=\sum_{i=0}^{p-1} b_i {\rm tr}(y^i)=-b_{p-1}$ and
$-b_{p-2}={\rm tr}(yx_0)=$
$$yx_0+(y-1)\cdot x_0 g+(y-2)\cdot x_0 {g^2}+\cdots+(y-(p-1))\cdot x_0 {g^{p-1}}=
y-y=0.$$
Hence $D_k^G=k[\sigma,b_0,\cdots,b_{p-3}]$ is a polynomial ring, as $D_k$ and $D_k^G$ have
Krull-dimension $p-1$.
\end{proof}

\begin{rem}\label{contrast_to_graded_case}
The result above is in marked contrast with the usual homogeneous algebra
$\mathrm{Sym}_k(kG)^G$, which is far more inscrutable (see for
example \cite{fssw} for the general construction of this ring of invariants).
\end{rem}
We obtain the following structure theorem for trace-surjective $G$-algebras with
prescribed ring of invariants:

\begin{thm}\label{struct_thm_Cp}
Let $G=\langle g\rangle\cong C_p$, $R$ an arbitrary commutative $k$-algebra
and $A\in \mathfrak{Ts}$ with $R=A^G$, then
$A\cong R[Y]/(Y^p-Y-\gamma)$ for suitable $\gamma\in R$ and $Yg=Y-1$.
\end{thm}
\begin{proof}
This is the special case of Theorem \ref{strct_thm_intro} with $n=1$, which will be proved in Section
\ref{poly_strct}.
\end{proof}

\begin{cor}$[${\rm Artin-Schreier}$]$\label{artin_schreier}
Let ${\rm char}\ k=p>0$. Suppose that $\mathbb{L}$ is a Galois extension of $k$ with Galois-group $G:=C_p$, then
there is $\gamma\in k$, such that $\mathbb{L}=k[\alpha]$ with
$\alpha$ being a root of the irreducible polynomial $X^p-X+\gamma\in k[X].$\\
Conversely, if the polynomial $f:=X^p-X+\gamma\in k[X]$ is irreducible and $\alpha\in \overline{k}$ is a root of
$f$, $\mathbb{L}:=k(\alpha)$ is the splitting field of $f$ and is a Galois extension of $k$ with Galois-group $G:=C_p$.
\end{cor}

\begin{rem}
If $\beta\in \mathbb{L}$ is any root of the irreducible polynomial $X^p-X+\gamma\in k[X]$,
then $\beta-i$ with $i=0,\cdots,p-1$ are the conjugates and ${\rm tr}(\beta^{p-1})=-1$. It follows
from Lemma \ref{rec_free}, that the conjugate elements $(\beta-i)^{p-1}$, $i=0,\cdots,p-1$ are linearly independent
and therefore form a normal basis of $\mathbb{L}$ over $k$.
\end{rem}

\section{General trace-surjective $G$-algebras}\label{ts_alg}

\begin{thm}\label{first_main}
Let $A$ be a trace-surjective $G$-algebra. Let
$\{a_i\ |\ i\in I\}$ be a $k$-basis of the ring of invariants $A^G$
and $\{w_g\ |\ g\in G\}$, a basis of $W\le A$ (with
$w_g:=(w_1)g$ and $1\in W \cong V_{reg}$). Then the following hold:
\begin{enumerate}
\item For every $0\ne a\in A^G$ we have $aW \cong W\cong V_{reg}$.
\item $A=A^G\cdot W= \oplus_{i\in I}a_i\cdot W =\oplus_{g\in G}A^G\cdot w_g.$
In particular, $A$ is a free $A^G$-module and a free $A^G[G]$ module of rank one.
\item For every $G$-stable proper ideal $\frak{J}\unlhd A$ the quotient ring $\bar A:= A/\frak{J}$ is again a ${\rm ts}$-algebra, with
$(A/\frak{J})^G \cong A^G/\frak{J}^G$.
\item Every ideal of $A^G$ is contracted from a
$G$-stable ideal in $A$. Conversely every $G$-stable ideal of $A$ is extended from
an ideal of $A^G$. In other words the mappings $\mathcal{I}\mapsto A\mathcal{I}$ and
$\frak{J}\mapsto \frak{J}\cap A^G$ are inverse bijections on the sets of
ideals of $A^G$ and $G$-stable ones of $A$.
\end{enumerate}
\end{thm}

\begin{proof} We use the equivalent conditions of Lemma \ref{tr}.\\
(i):\ Since $W\cong kG$, we have $W^G \cong k$ and we can choose
the basis $\{w_g\ |\ g\in G\}$ such that $1 = \sum_{g\in G}\
(w_1)g = {\rm tr}(w_1)$. Assume $a\in A^G$ such that $aW\not\cong W$,
then $aW\cong W/X$ with $0\ne X\le W$, so $W^G={\rm tr}(W)\le X$ and
${\rm tr}(W/X) \le W^GX/X \le X/X=0$. Hence
$$a = a\cdot 1 = a\cdot {\rm tr}(w_1) = {\rm tr}(a\cdot w_1) =0.$$
(ii):\ Since for all $i$ we have $a_iW \cong W$, $(a_iW)^G=k\cdot a_i$.
It follows from Lemma \ref{observation1}, that
$A^G\cdot W = \oplus_{i\in I}k\cdot a_i\cdot W$ with each $a_i\cdot
W \cong V_{reg}$ and again this is an injective module by H. Bass'
theorem. Hence $A = A^G\cdot W \oplus C$ with some complementary
free $kG$ - module $C$. However $C^G \le A^G \le A^G\cdot W$,
since $1\in W$ and therefore $C^G=0$. Thus $C$ is a free
$kG$-module not containing an invariant, hence $C=0$.
This shows that $A^G\cdot W=A$, from
which the second direct sum decomposition immediately follows.\\
(iii):\ Let $\frak{J}\unlhd A$ be a $G$-stable proper ideal.
Then ${\rm tr}(\bar w_1) = \sum_{g\in G} \bar w_1 g = \bar 1$, so
$A/\frak{J}$  is a ${\rm ts}$-algebra. Let $\bar x\in (A/\frak{J})^G$,
then $\bar x = 1\cdot \bar x = {\rm tr}(w_1)\bar x = \sum_{g\in G} w_1g\cdot
\bar x =$
$\sum_{g\in G} w_1g\cdot \bar x g=\sum_{g\in G}
(\overline{w_1x})g = {\rm tr}(\overline{w_1x}) = \overline{{\rm tr}(w_1x)} \in
A^G/\frak{J}^G.$ Hence $A^G/\frak{J}^G = (A/\frak{J})^G$.\\
(iv): Let $\mathcal{I}\le A^G$ be an ideal and $x = \sum_\ell a_\ell
i_\ell\in A\mathcal{I}\cap A^G$. Then $x=1\cdot x = {\rm tr}(w_1x)=\sum_\ell {\rm tr}(w_1a_\ell i_\ell)=$
$\sum_\ell {\rm tr}(w_1a_\ell) i_\ell \in A^G\mathcal{I}=\mathcal{I},$ hence $A\mathcal{I}\cap A^G=\mathcal{I}$. Now let $\frak{J}\unlhd A$ be a
$G$-stable ideal, let $\overline{()}\ :\ A \to \overline A:=A/\frak{J}$ denote
the $G$-equivariant canonical epimorphism and choose the basis $\{a_i\ |\ i\in I\}$, such
that it extends a basis for $\frak{J}^G$ over $k$. Applying Theorem \ref{first_main} (i) to $A/\frak{J}$
and using \ref{first_main} (ii) we get $\frak{J}=\oplus_{a_i\in \frak{J}^G\atop i\in I} a_iW\subseteq \frak{J}^GA.$
\end{proof}

\begin{rem}\label{main_p_nc_rem}
The results in \ref{first_main} also hold for non-commutative $k$-algebras
and one or two-sided ideals, respectively. There is also a version of
\ref{first_main} for arbitrary finite groups $X$:
Let $B$ be a (not necessarily commutative) trace surjective $X$-algebra,
$\{a_i\ |\ i\in I\}$ a $k$-basis of $B^X$ and
$\{w_j\ |\ j=1,\cdots,s\}$ a basis of $W\le B$ with $1\in W
\cong P(k)$, the projective cover of the trivial $kX$-module.
Then
\begin{enumerate}
\item $B=B^X\cdot W \oplus C$ with $B^XW = \oplus_{i\in I}a_i\cdot W = \oplus_{j=1}^sB^X\cdot w_j$
and $C$ is a projective $kX$-module not containing a summand $\cong P(k)$.
In particular, $B^XW$ is a free $B^X$-module.
\item For every $X$-stable proper ideal $\frak{J}\unlhd\ _BB$ the $kX$-module $B/\frak{J}$ is projective
and we have $(B/\frak{J})^X \cong B^X/\frak{J}^X$. For every $X$-stable two-sided proper
ideal $I\unlhd B$ the quotient ring $\bar B:= B/I$ is again a trace-surjective
$X$-algebra.
\end{enumerate}
\end{rem}

\begin{df}\label{D_k}
Let $V_{reg}:=\oplus_{g\in G} kX_g$ be the regular $kG$-module.
Define:
$$D_k:=D_k(G):= {\rm Sym}_k(V_{reg})/(\alpha)$$
with $\alpha = -1+\sum_{g\in G} X_g.$
\end{df}
Setting $x_g:=X_g+(\alpha)$, it follows that $D_k\cong k[x_g\ |\ 1\ne g\in G]$
is a polynomial ring of Krull-dimension ${\rm Dim}(D_k)=|G|-1.$
Moreover, for $t:={\rm tr}(X_1)$ the map $$D_k\to ({\rm Sym}_k(V_{reg})[1/t])_0;\ x_g\mapsto X_g/t$$
is a $G$-equivariant isomorphism of $G$-algebras. In other words,
$D_k$ is the degree zero component of the homogeneous localization of
${\rm Sym}_k(V_{reg})$ at the $G$-invariant $t$, or the $G$-equivariant ``dehomogenization of
${\rm Sym}_k(kG)$" as described in \cite{BH} (Proposition 1.5.18 pg.38).
In particular ${\rm Sym}_k(V_{reg})[1/t]\cong D_k[t,1/t]$ and
$$({\rm Sym}_k(V_{reg})[1/t])^G\cong {\rm Sym}_k(V_{reg})^G[1/t]\cong D_k^G[t,1/t].$$

The algebra $D_k$ is ``universal" in the sense that every trace-surjective algebra
can be constructed from quotients of $D_k$ by extending invariants:

\begin{prp}\label{D_k_gen_1}
Let $A\in \mathfrak{Ts}$ with point $a\in A$, generating the
$kG$-subspace $W:=\langle ag\ |\ g\in G\rangle_k\le A$. Then there is an epimorphism
of $G$-algebras
$$\theta:\ D_k \to k[W]\le A$$ such that
$A = A^G\otimes_{k[W]^G} k[W]$. The algebras $D_k$ and $k[W]$ are
free $kG$ - modules.
\end{prp}

\begin{proof} Let $S$ be the subalgebra $S:=k[W]\le A$.
Then by Theorem \ref{first_main}, $S =S^G\cdot W\cong \oplus_{g\in G}S^G w_g$ is a free
$kG$-module. Let $\phi$ be the canonical extension of the map $X_g\mapsto w_g$
to a $k$-algebra homomorphism of ${\rm Sym}_k(V_{reg})$ onto $S$, then clearly the ideal
$(\alpha)$ is contained in ${\rm ker}(\phi)$, so $S$ is a quotient of $D_k$. Moreover
$$A^G\otimes_{S^G}S \cong \oplus_{g\in G} A^G \otimes_{S^G} S^G w_g
\cong \oplus_{g\in G} A^G w_g = A.$$
\end{proof}

Note that the $W$ in Proposition \ref{D_k_gen_1} can be chosen to be any submodule
isomorphic to $V_{reg}$ with $1_A\in W$, and the subalgebra $k[W]\le A$ is isomorphic to
$D_k/I'$ with a suitable $G$-stable ideal $I'\unlhd D_k$. Conversely,
for any $G$-stable ideal $I\unlhd D_k$ with quotient $S:=D_k/I$, and any extension
of commutative $k$-algebras $T\ge S^G$, the tensor product
$A:=T\otimes_{S^G}S$, with $G$-action only on the right tensor factor,
is a trace-surjective $G$-algebra with $A^G=T$.
In other words, $A$ is an ``extension by invariants" of a ``cyclic" ${\rm ts}$-algebra
of the form $D_k/I$ with $I\unlhd D_k$ a $G$-stable ideal.

We end this section by pointing out a connection between the concepts of
commutative trace-surjective algebras and Galois extensions of
commutative rings in the sense of \cite{chr}.
Let $\Gamma$ be an arbitrary finite group, $A$ be a $\Gamma$-algebra, $R:=A^\Gamma$ and let $\Gamma*A$ be the trivially crossed
group ring, i.e. the ``semidirect product" $\Gamma\times A$ with multiplication
$ga \cdot g'a':=gg' (a)g'a'$.
Note that $\Gamma*A$ acts naturally on $A$ by $R$-algebra homomorphisms, giving rise
to a homomorphism
$$j:\ \Gamma*A\to {\rm End}_R(A),\ ga\mapsto (a'\mapsto (aa')g^{-1}).$$

\begin{thm}(Chase-Harrison-Rosenberg \cite{chr})\label{chase_harrison_rosenberg}
The following statements are equivalent:
\begin{enumerate}
\item There are elements $x_1,\cdots,x_n$, $y_1,\cdots,y_n$ in $A$, such that
$\sum_{i=1}^n x_i(y_i)\sigma=\delta_{1,\sigma}\ \forall\sigma\in \Gamma.$
\item $A$ is a finitely generated projective $R$-module and
$j:\ \Gamma*A\to {\rm End}_R(A)$ is an isomorphism.
\item For every $1\ne \sigma\in \Gamma$ and maximal ideal ${\rm p}$ of $A$ there
is $a\in A$ with $a-(a)\sigma\not\in{\rm p}$.
\end{enumerate}
\end{thm}
If $R=A^\Gamma\le A$ satisfies any of these conditions, then the ring extension $R\le A$ is
called a \emph{Galois-extension}. It follows easily from condition (i) in
Theorem \ref{chase_harrison_rosenberg}, that
$\sum_{i=1}^n x_i{\rm tr}(y_i)=1$, which implies ${\rm tr}(A)=R$ (see \cite{chr} Lemma 1.6),
so every Galois extension is trace-surjective. For arbitrary finite groups, not
every trace-surjective algebra needs to be a Galois-extension, but in the
case of a finite $p$-group $G$ in characteristic $p$, we have:

\begin{prp}\label{G_p_gr_ts_iff_gal}
Let $A$ be a $G$-algebra. Then $A$ is trace-surjective, if and only if $A^G\le A$ is
a Galois-extension (in the sense of \cite{chr}).
\end{prp}
\begin{proof}
It remains to show that a ${\rm ts}$-algebra is a Galois-extension and by
Theorem \ref{chase_harrison_rosenberg} (iii) it suffices to show that
$A^G\hookrightarrow A$ is unramified.\\
Let ${\mathfrak P}\in {\rm Spec}(A)$ with residue class
field $L:= {\rm Quot}(A/{\mathfrak P})$, decomposition group
$G_{\mathfrak P}:= {\rm Stab}_G({\mathfrak P})$, inertia group
$T_{\mathfrak P}:= \{g\in G_{\mathfrak P}\ |\ g\ {\rm acts\ trivially\ on\ L}\} \unlhd G_{\mathfrak P}$
and $C:=A/{\mathfrak P}$.
Let $1_L$ be the unit in $L$, then
$1_L=1_C=1_C\cdot 1_{A}=1_C\cdot {\rm tr}(a_1)=$
$$1_C\sum_{x\in G/G_{\mathfrak P}}\sum_{y\in G_{\mathfrak P}} (a_1)yx =
1_C\sum_{x\in G/G_{\mathfrak P}}\sum_{y\in G_{\mathfrak P}} (a_1)x^{-1}y^{-1}=
1_C\sum_{y\in G_{\mathfrak P}} (\tilde s)y = \tr^{(G_{\mathfrak P})}(1_C\tilde
s)=$$
$$t_{T_{\mathfrak P}}^{G_{\mathfrak P}}\circ
\tr^{(T_{\mathfrak P})}(1_C\tilde s)=|T_{\mathfrak P}|\cdot t_{T_{\mathfrak P}}^{G_{\mathfrak P}}(1_C\tilde s),$$
so we conclude that $T_{\mathfrak P}=1$.
(Here $t_{T_{\mathfrak P}}^{G_{\mathfrak P}}$ denotes an obvious ``relative transfer map").
\end{proof}

\begin{cor}\label{alg_grps}
Let $k$ be algebraically closed and $X$ an affine algebraic group with finite
subgroup $\Gamma$. Then the ring $k[X]$ of regular functions is a trace-surjective
$\Gamma$-algebra.
\end{cor}
\begin{proof}
The right regular action of $\Gamma$ on $X={\rm max-Spec}(k[X])$ is fixed point free, hence the claim follows from Theorem \ref{chase_harrison_rosenberg} (iii) and the remark after that Theorem (or Proposition \ref{G_p_gr_ts_iff_gal}, if $\Gamma$ is a $p$-group).
\end{proof}

\begin{rem}\label{Group_functors}
There is a more general version of Corollary \ref{alg_grps}: assume that $\Gamma$ is realized as a
subgroup scheme of an affine $k$-group scheme $X$, then the algebra $k[X]$
is also trace-surjective. This follows from the fact that in this situation
$k[X]$ is an injective module of $\Gamma$ as $k$-functor (\cite{jcj} 5.5.(6),5.13 and 4.12).
The description of injective modules in \cite{jcj} 3.10 shows that $k[X]$ is a
projective $k\Gamma$-module. It is easy to see that Lemma \ref{tr} is valid, if
$G$ is replaced by $\Gamma$ and ``free" by ``projective", hence it follows that
$k[X]$ is a trace-surjective $\Gamma$-algebra.
\end{rem}

\section{Standard subalgebras}\label{std_sub}


Let $A$ be a trace-surjective $G$-algebra; then as seen in Proposition \ref{D_k_gen_1},
$A\cong A^G\otimes_{(D_k/I)^G} D_k/I$ with some $G$-stable ideal
$I\unlhd D_k$. In that sense $D_k$ is a ``universal model" for trace-surjective
$G$-algebras. However, if $S$ is any trace-surjective subalgebra
$S\le D_k$, then $D_k\cong D_k^G\otimes_{S^G} S$, hence
$$A\cong A^G\otimes_{(S/J)^G} S/J$$ with $S/J\le A$ and $G$-stable ideal
$J\unlhd S$, so $S$ is also ``universal". The subalgebras $S\le D_k$ which are also
retracts turn out to be particularly useful. This motivates the following

\begin{df}\label{standard_universal}
A trace-surjective  $G$-subalgebra $U\le D_k$ will be called {\bf standard},
if $D_k=U\oplus J$, where $J$ is some $G$-stable ideal.
\end{df}

If $U\le D_k$ is standard, then there is a projection morphism
$\chi:\ D_k\to U\hookrightarrow D_k$, which is an idempotent $k$-algebra endomorphism of $D_k$.
More precisely:

\begin{lemma}\label{endo_krit_standard}
Let $U\le D_k$ be a trace-surjective  $G$-algebra, then the following are equi\-valent:
\begin{enumerate}
\item $U$ is standard.
\item $\exists\ \chi=\chi^2\in ({\rm End}_{k-{\rm alg}}(D_k))^G$ with $U=\chi(D_k)$.
\item $\exists\chi\in ({\rm End}_{k-{\rm alg}}(D_k))^G$ with
$\chi^2(x_1)=\chi(x_1)=:w\in U=k[wg\ |g\in G]$.
\item $\exists\ w=W(x_1,x_{g_2},\cdots,x_{g_{|G|}})\in U$ with ${\rm tr}(w)=1$,
$w=W(w,wg_2,\cdots,w_{g_{|G|}})$ and $U=k[wg\ |g\in G]\le D_k.$
\end{enumerate}
\end{lemma}
\begin{proof} ``{\rm (i)} $\Rightarrow$ {\rm (ii)}": If $D_k=U\oplus I$, choose
$\chi$ to be the canonical projection $\chi:\ D_k\to U.$ \\
``{\rm (ii)} $\Rightarrow$ {\rm (iii)}": $U=\chi(D_k)= \chi(k[x_g\
|g\in G])=k[\chi(x_1)g\ |g\in G]=k[wg\ |g\in G].$\\
``{\rm (iii)} $\Rightarrow$ {\rm (iv)}": Clearly
$\chi^2(x_1)=\chi(x_1)$ implies $\chi^2=\chi$. Hence $w:=\chi(x_1)=
\chi^2(x_1)=\chi(w)=\chi(W(x_1,x_{g_2},\cdots,x_{g_{|G|}}))=$
$$W(\chi(x_1),\chi(x_{g_2}),\cdots,\chi(x_{g_{|G|}}))=
W(w,wg_2,\cdots,w_{g_{|G|}}).$$ Clearly ${\rm tr}(w)=1$.\\
``{\rm (iv)} $\Rightarrow$ {\rm(i)}": Since ${\rm tr}(w)=1$, the map
$x_g\mapsto w_g$ extends to a $G$-equivariant epimorphism of
$k$-algebras $\theta:\ D_k\to U$, satisfying
$\theta(w)=W(\cdots,wg\cdots)=w$. Hence $\theta(wg)=\theta(w)g=wg$
and we conclude that $\theta_{|U}={\rm id}.$ It follows that
$D_k=U\oplus \ker(\theta).$
\end{proof}

Let $U\le D_k$ be standard. Since $D_k$ is a polynomial ring
it follows from \cite{costa} Corollary 1.11, that $U$ is a regular UFD. It is apparently
still an open question, whether or not every retract of a polynomial ring is again
a polynomial ring (see also \cite{shpilrain} pg 481).

\begin{df}\label{reflexive_points}
A point $w\in D_k$ will be called {\bf reflexive}, if
$$w=W(x_1,\cdots,x_g\cdots)=W(w,\cdots,wg,\cdots)=\theta(w),$$
where $\theta\in ({\rm End}_{k-{\rm alg}}(D_k))^G$ is defined by
$x_g\mapsto w\cdot g$ $\forall g\in G$.
\end{df}

\begin{rem}\label{refl_pt_rem}
\begin{enumerate}
\item By definition a trace-surjective  $G$-algebra is cyclic, if and only if it is generated as an algebra by the
$G$-orbit of one point. Lemma \ref{endo_krit_standard} shows, that the standard subalgebras of $D_k$ are precisely the subalgebras generated by the $G$-orbit of a \emph{reflexive} point.
\item There are categorical characterizations of $D_k$ and its standard subalgebras.
It turns out that $D_k$ is a projective generator in $\mathfrak{Ts}$ and
the standard subalgebras are precisely the cyclic projective objects.
We will investigate this and further properties of $\mathfrak{Ts}$ in
a subsequent paper.
\end{enumerate}
\end{rem}

Before describing a general inductive procedure to construct reflexive points
and standard subalgebras of $D_k$ here are some examples:

\begin{ex}\label{exx_cp}
Let $G:=\langle g\rangle \cong C_p$ be cyclic of order $p>2$, then
the subalgebra $k[y]\le D_k$ of section \ref{ex_cp} is standard
with reflexive point $-y^{p-1}$ where $y:=\sum_{i=1}^{p-1} ix_i\in D_k$.
\end{ex}

\begin{ex}\label{extraspecial_p_cube}
Let $G=\langle g_1,g_2\ |\ g_1^p=g_2^p=[g_1,g_2]^p=1\rangle$ be extraspecial
of order $p^3$, exponent $p>2$ and centre
$Z:=Z(G) \cong \langle g_0\rangle$ with $g_0:=[g_1,g_2]$.
Every element of $G$ can be uniquely written in the
form $g=g_0^{a_0}g_1^{a_1}g_2^{a_2}$, with $a_i\in \mathbb{F}_p$. Now
define as before, $V_{reg}:=\oplus_{g\in G}kX_g \cong kG$, the
regular $kG$-module, $D:=D_k:={\rm Sym}_k(V_{reg})/J=k[x_g\ |g\in
G]$ with $J=(j){\rm Sym}_k(V_{reg})$, $j:=-1+\sum_{g\in G} X_g$,
$x_g:=X_g+J$ and $\sum_{g\in G} x_g=1$. For $i=0,1,2$ set
$$y_i:=\sum_{\underline a\in \mathbb{F}_p^3} a_i \cdot x_{g_0^{a_0} g_1^{a_1} g_{2}^{a_{2}}}\in D_k.$$
Then a straightforward calculation using Lemma \ref{cycl_p_lemma} shows, that
$$y_{2}\cdot g_{2}=\sum_{\underline a\in \mathbb{F}_p^3} a_{2}\cdot x_{g_0^{a_0}g_1^{a_1}g_{2}^{a_{2}}}\cdot g_{2}=
\sum_{\underline a\in \mathbb{F}_p^3} (a_{2}+1) x_{(g_0^{a_0} g_{1}^{a_1}g_{2}^{a_2+1})} -
\sum_{\underline a\in \mathbb{F}_p^3} x_{(g_0^{a_0}g_{1}^{a_1})g_{2}^{a_{2}+1}}=y_{2}-1,$$
and $y_j\cdot g_{2}=y_j$ for $j<2$.
Using the formula $[x^a,y^b]=[x,y]^{ab}$ for $x,y\in G$ (since $[x,y]\in Z(G)$)
one has
$$y_j\cdot g_1=\sum_{\underline a\in \mathbb{F}_p^3}
a_j\cdot x_{g_0^{a_0}g_1^{a_1}g_{2}^{a_{2}}}\cdot g_1=\sum_{\underline a\in \mathbb{F}_p^3} a_j \cdot x_{g_0^{a_0-{a_{2}}}\cdot g_{1}^{a_{1}+1} \cdot g_{2}^{a_{2}}} .$$
Hence $y_1\cdot g_1=y_1-1$, $y_2\cdot g_1=y_2$, and for $j=0$ we get
$$y_0\cdot g_1=\sum_{\underline a\in \mathbb{F}_p^3} (a_0-a_{2}) x_{g_0^{a_0-{a_{2}}}\cdot g_{1}^{a_{1}+1} \cdot g_{2}^{a_{2}}}+
a_{2}\cdot x_{g_0^{a_0-{a_{2}}}\cdot g_{1}^{a_{1}+1} \cdot g_{2}^{a_{2}}}=y_0+y_{2}.$$
Let ${\mathcal B}:= \{y_0,y_1,y_2,1\}$, then we see that the subspace $\langle {\mathcal B}\rangle\le D_k$ is
$G$-stable, providing a faithful representation of $G$ with
$$M_{\mathcal B}(g_1)=
\left( \begin{array}{cccc}
\ 1&0&1&\ 0\\
\ 0&1&0&-1\\
\ 0&0&1&\ 0\\
\ 0&0&0&\ 1\\
\end{array}\right)\ {\rm and}\ M_{\mathcal B}(g_2)=
\left( \begin{array}{cccc}
\ 1&0&0&\ 0\\
\ 0&1&0&\ 0\\
\ 0&0&1&-1\\
\ 0&0&0&\ 1\\
\end{array}\right).$$
It is not hard to see that
\begin{itemize}
\item the elements $y_0,y_1,y_2\in D_k$ are algebraically independent.
\item the subspace $Y:=\langle y_0^{z_0}y_1^{z_1}y_{2}^{z_{2}}\ |\ 0\le z_i<p\rangle\le D_k,$
is a free $kG$-module of rank one with
${\rm tr}((y_0y_1y_{2})^{p-1})= -1.$
\item The $G$-algebra endomorphism $\theta:\ D_k\to D_k$ defined by the map
$x_g \to wg$ with $w:=-(y_0y_1y_{2})^{p-1}$ satisfies
$\theta(y_i)=y_i$ for $i=0,1,2$.
\end{itemize}
Hence the polynomial ring $U:=k[y_0,y_1,y_{2}]\le D_k$ is a standard
$G$-subalgebra of $D_k$ with reflexive point $w=-(y_0y_1y_{2})^{p-1}$.
\end{ex}

\begin{ex}\label{direct_product}
Consider the $p$-groups $G$, $H$ and $G\times H$ with
$D_k(G)=k[x_g\ |\ g\in G]$, $D_k(H)=k[y_h\ |\ h\in H]$ and
$D_k(G\times H)=k[t_{gh}\ |\ gh\in G\times H]$ and
$$\sum_{g\in G}x_g=1,\ \sum_{h\in H}y_h=1,
\sum_{gh\in G\times H}t_{gh}=1, {\rm respectively}.$$
For $g\in G$ and $h\in H$ define $X_g:=\sum_{h\in H}t_{gh}$ and
$Y_h:=\sum_{g\in G}t_{gh}$. Then
it easy to see that $\omega:=X_1Y_1\in D_k(G\times H)$ is a point. Define
a $G\times H$-algebra homomorphism by
$$\theta:\ D_k(G\times H)\to D_k(G\times H),\ t_e gh\mapsto (\omega) gh,$$
then an elementary calculation shows that $\theta(\omega)=\omega$,
hence $\omega$ is a reflexive point in the subalgebra
$S:=k[X_g,Y_h\ |\ 1\ne g\in G, 1\ne h\in H]\le D_k(G\times H)$.
Obviously the $X_g,Y_h$ with $1\ne g,h$ are algebraically independent
and the map
$D_k(G)\otimes_k D_k(H)\to D_k(G\times H)$ extending $x_g\otimes y_h\mapsto X_gY_h$,
is a $G\times H$-equivariant $k$-algebra homomorphism. It maps
the (generating) point $x_1\otimes y_1\in D_k(G)\otimes_k D_k(H)$ to $\omega$
and it maps $D_k(G)\otimes_k D_k(H)$ isomorphically onto $S=k[\omega^{G\times H}]$.
It follows that the algebra $D_k(G)\otimes D_k(H)$ is a standard subalgebra of $D_k(G\times H)$
of Krull-dimension $|G|+|H|-2$.
\end{ex}

We are now going to describe an inductive procedure to construct reflexive points and
standard subalgebras in $D_k(G)$ for arbitrary $p$-groups $G$:
Let $F=\oplus_{g\in G} kx_g\cong kG$, $F\le D_k$ and let
$Z:=\langle g_0\rangle\cong C_p$ be a subgroup of the centre of $G$.
For each $\bar g:=gZ\in G/Z$ we define
\begin{equation}\label{def_eta_g}
\eta_g:=\eta_{\bar g}:=\sum_{h\in gZ} x_h\in F^Z,
\end{equation}
then the $\eta_{\bar g}$ for $\bar g\ne \bar 1$ are algebraically independent,
$\sum_{\bar g\in G/Z}\eta_{\bar g}=1$ and
$\sum_{\bar g\in G/Z}k\eta_{\bar g}=\oplus_{\bar g\in G/Z}k\eta_{\bar g}\cong k(G/Z).$
Hence
\begin{equation}\label{D_k_G_Z}
D_k(G/Z)\cong \tilde D:=k[\eta_{\bar g}|\bar g\in G/Z] \le D_k(G)^Z.
\end{equation}
Let $G=\uplus_{r\in {\mathcal R}}Zr$ with ${\mathcal R}$ a transversal of
$Z$-cosets in $G$. Set $t_{\mathcal R} \in {\rm End}_k(D_k(G))$, mapping
$f\mapsto \sum_{r\in {\mathcal R}}\ fr$ and define
$y_0:=\sum_{i=1}^{p-1}i\cdot t_{\mathcal R}(x_{g_0^i}).$
For any $h\in G$ we have
$$y_0\cdot h = y_0 - \sum_{r\in {\mathcal R}} e_{r,h}\cdot \eta_{r_{r,h}},$$
where $e_{r,h}\in \mathbb{F}_p$ and $r_{r,h}\in {\mathcal R}$ are defined by
the equation
\begin{equation}\label{def_e_r_h}
rh=g_0^{e_{r,h}}r_{r,h}.
\end{equation}
Moreover, for every point $w'\in D_k(G/Z)$, the element
\begin{equation}\label{w_def}
w:=-(y_0)^{p-1}\cdot w'\in D_k(G)
\end{equation}
is a point. Indeed: we have
$y_0\cdot h =\sum_{i=0}^{p-1}i\cdot \sum_{r\in {\mathcal R}}x_{g_0^irh} =
\sum_{i=0}^{p-1}i\cdot \sum_{r\in {\mathcal R}}x_{g_0^{i+e_{r,h}}r_{r,h}}=$
$$\sum_{r\in {\mathcal R}}\big(\sum_{i\in \mathbb{F}_p}(i+e_{r,h})x_{g_0^{i+e_{r,h}}r_{r,h}}-
e_{r,h}\sum_{i\in \mathbb{F}_p} x_{g_0^{i+e_{r,h}}r_{r,h}}\big)=$$
$$\sum_{r\in {\mathcal R}}\big(\sum_{i\in \mathbb{F}_p}ix_{g_0^ir_{r,h}}-
e_{r,h}\sum_{i\in \mathbb{F}_p} x_{g_0^ir_{r,h}}\big)=
y_0-\sum_{r\in {\mathcal R}}e_{r,h}\cdot \eta_{r_{r,h}}.$$
Since $e_{r,g_0}=1$ and $r_{r,g_0}=r$ for every $r\in {\mathcal R}$,
it follows that
$y_0g_0=y_0-1.$
Let $w:=-y_0^{p-1}\cdot w'$ and $t_Z^G$ be the relative transfer map,
then we get
$${\rm tr}(w)=t_Z^G\circ \tr^{(Z)}(-w'\cdot y_0^{p-1}) =
-t_Z^G(w'\cdot \tr^{(Z)}(y_0^{p-1})) = -\tr^{(G/Z)}(w'\cdot(-1))=1,$$
and $w$ is a point, but in general not a \emph{reflexive} point.

\begin{lemma}\label{tryout}
Assume that $w'$ is a reflexive point in $D_k(G/Z)$, with standard subalgebra
$$U_{G/Z}:=k[w'\bar g\ |\bar g\in G/Z]\le D_k(G/Z).$$
Let $\theta\in ({\rm End}_{k-{\rm alg}}(D_k))^G$ be defined by $x_g\mapsto w\cdot g$ with
point $w:=-y_0^{p-1}\cdot w'$. Then the following hold:
\begin{enumerate}
\item Set $\tilde y_0=\theta^{p-1}(y_0)$, then
the element $$\tilde w:= \theta^p(x_1)=\theta^{p-1}(w)=-\tilde y_0^{p-1}w'$$ is a reflexive point of $D_k$,
generating a standard subalgebra
$$U_G:=k[\tilde wg\ |\ g\in G]=k[\tilde y_0,w'{\bar g}\ |\bar g\in G/Z]=
k[\tilde y_0,U_{G/Z}]=k[\theta(y_0),U_{G/Z}],$$
which is a subalgebra of the $G$-subalgebra
$k[y_0,D_k(G/Z)]=k[\tilde y_0,D_k(G/Z)].$
\item The $G$-action on $U_G$ is determined by the one on $U_{G/Z}$ and the formula
\begin{equation}\label{g_tilde_y0}
\tilde y_0g=\tilde y_0-\sum_{r\in {\mathcal R}}e_{r,g}\cdot (w'rg)\ \forall g\in G.
\end{equation}
\item The element $y_0$ is algebraically independent of $D_k(G/Z)$ and
$\tilde y_0$ is algebraically independent of $U_{G/Z}.$
\end{enumerate}
\end{lemma}

\begin{proof} Recall the definition of the $\eta_{\bar g}$ in equation (\ref{def_eta_g}); we
define $k$-algebra endomorphisms $\theta'\in ({\rm End}_{k-{\rm alg}}(D_k(G/Z)))^{G/Z}$ by $\theta'(\eta_{\bar g}):= w'\cdot \bar g$
and $\theta\in ({\rm End}_{k-{\rm alg}}(D_k(G)))^G$ by\\
$\theta(x_g):= w\cdot g.$
Since we assume now that $w'$ is a reflexive point of $D_k(G/Z)$, we have $\theta'(w')=w'.$
We have
$\theta(\eta_{\bar g})= \sum_{i\in \mathbb{F}_p} \theta(x_{g_0^ig})=$
$$\sum_{i\in \mathbb{F}_p} (-y_0^{p-1}w')g_0^ig=
\sum_{i\in \mathbb{F}_p} \big([(-y_0^{p-1})g_0^i]\cdot w'\big)g=w'g=\theta'(\eta_{\bar g}).$$
It follows that $\theta(w')=w'$.
Moreover, using Lemma \ref{cycl_p_lemma} we get
$$\theta(y_0)=\sum_{i\in \mathbb{F}_p\atop r\in {\mathcal R}} i(-y_0^{p-1}w')g_0^ir=
\sum_{i\in \mathbb{F}_p\atop r\in {\mathcal R}} [i(-y_0^{p-1}g_0^i)w']r=
\sum_{r\in {\mathcal R}} (y_0w')r,\ {\rm if}\ p>2$$
and $\sum_{r\in {\mathcal R}} [(y_0+1)w']r=$
$\sum_{r\in {\mathcal R}} (y_0w')r+\sum_{r\in {\mathcal R}} w'r=$
$1+\sum_{r\in {\mathcal R}} (y_0w')r$, if $p=2$.
\footnote{In particular $\theta(w)=w\iff \theta(y_0)^{p-1}=y_0^{p-1}\iff$
$\theta(y_0)=\lambda\cdot y_0$ for some $0\ne \lambda\in \mathbb{F}_p.$}\\
Recall the definition of the $e_{r,h}\in \mathbb{F}_p$ and $r_{r,h}\in {\mathcal R}$ in
equation (\ref{def_e_r_h}). We have
$\theta(y_0)=$
$(1+)\sum_{r\in {\mathcal R}} (y_0r)(w'r)= (1+)\sum_{r\in {\mathcal R}}
(y_0-\sum_{r'\in {\mathcal R}}e_{r',r}\cdot \eta_{r_{r',r}})(w'r)=$
$y_0-\zeta$, where
$\zeta=(1+) \sum_{r,r'\in {\mathcal R}}e_{r',r}\cdot \eta_{r_{r',r}}(w'r)\in D_k(G/Z)$,
but not necessarily in $U_{G/Z}$. (The summand ``$1+$" appears only if $p=2$).
However, using $\theta(\eta_{r_{r',r}})=\theta'(\eta_{r_{r',r}})=w'r_{r',r}=w'r'r$ and
$\sum_{r\in {\mathcal R}}w'r=1$
we get
$\theta^2(y_0)=\theta(\theta(y_0))=$
$$(1+) \sum_{r\in {\mathcal R}}
(\theta(y_0)-\sum_{r'\in {\mathcal R}}e_{r',r}\cdot w'r'r)(w'r)=\theta(y_0)+\gamma$$
with $\gamma:=(1+)-\sum_{r\in {\mathcal R}}
[\sum_{r'\in {\mathcal R}}e_{r',r}\cdot (w'r')w']r\in U_{G/Z}.$
Hence we have $\gamma=\theta(\gamma)$ and we conclude that $\theta^i(\theta(y_0))=\theta(y_0)+i\cdot\gamma$, so
$\theta^p(\theta(y_0))=\theta(y_0)$.
It follows that $\theta(w)=\theta^2(x_1)$ is fixed under $\theta^p$, as $\theta(w')=w'$, hence
$$(\theta^p)^2(x_1)=\theta^{p-2}(\theta^p(\theta^2(x_1))) =
\theta^{p-2}(\theta^2(x_1))=\theta^p(x_1).$$ Let $\chi:= \theta^p$
and $\tilde w=\theta^{p-1}(w)= \chi(x_1),$ then $\chi(\tilde
w)=\tilde w$ and it follows from Lemma \ref{endo_krit_standard}, that
$\tilde w$ is a reflexive point, whose $G$-orbit generates a standard
subalgebra
$$U_G:=\chi(D_k(G))=k[\tilde w\cdot g\ |g\in G].$$
Define $\tilde y_0:=\theta^{p-1}(y_0)$, then
$\tilde w=-\tilde y_0^{p-1}\cdot w'\in k[\tilde y_0,U_{G/Z}].$
For every $h\in G$ we have $\tilde y_0h=$
$\tilde y_0-\sum_{r\in {\mathcal R}}e_{r,h}\cdot \theta^{p-2}(w'rh)=
\tilde y_0-\sum_{r\in {\mathcal R}}e_{r,h}\cdot (w'rh),$
so $k[\tilde y_0,U_{G/Z}]$ is $G$-stable and contains $U_G$.
On the other hand, $\tilde y_0= \chi(\tilde y_0)\in \chi(D_k(G))=U_G$ and likewise
$w'=\chi(w')\in U_G$, hence
$U_G=k[\tilde y_0,U_{G/Z}].$\\
For the last statement, recall that $\theta(y_0)=y_0-\zeta$ with
$\zeta \in D_k(G/Z)$, hence $\tilde y_0=y_0+\mu$ for some $\mu\in
D_k(G/Z)$ and $k[y_0,D_k(G/Z)]=k[\tilde y_0,D_k(G/Z)]$.
It therefore suffices to show that the set $S = \{y_0,
\eta_{\bar r}\ |r\in {\mathcal R}\backslash \{1\}\}$ is
algebraically independent. It is contained in $D_k(G)$, which is a
polynomial ring generated by the algebraically independent set
$\{x_{g_0^ir}\ne x_1\}$, and is clearly linearly independent over $k$.
Since $S \subseteq D_k(G)_1$ (the graded component of degree 1) it
follows directly that it is algebraically independent over $k$, as
required.
\end{proof}

\begin{cor}\label{direct product}
Let $G=Z\times Q$ be a $p$-group with $Z\cong C_p$ and $p>2$.
If $w'\in D_k(Q)$ is a reflexive point, then so is $w=-y_0^{p-1}w'\in D_k(G)$. If $U_Q=k[w'^Q]\le D_k(Q)$
is a standard subalgebra, then $U:=k[w^G]\le D_k(G)$ is standard.
\end{cor}
\begin{proof}
In this case we can choose ${\mathcal R}:=Q$ and then have $e_{r,q}=0$ for all $r,q\in Q$.
Hence $y_0\in D_k(G)^Q$ and the proof of Lemma \ref{tryout} shows that
$\theta(w)=w$ is a reflexive point.
\end{proof}

\begin{rem}\label{direct product_rem}
If $p=2$, $w$ is not a \emph{reflexive} point, but $w+w'=\theta^2(x_1)$ is.
\end{rem}

\begin{thm}\label{arb_p_grp}
Let $G$ be an arbitrary finite $p$-group of order $p^n$. Then there is a trace-surjective
standard $G$-subalgebra $U\le D_k(G)$, such that $U$ is a polynomial ring
of Krull-dimension $n$.
\end{thm}

\begin{proof} The proof is by induction on $n$, where the case $n=1$ follows from Proposition \ref{cycl_p_univ_model}.
With notation from Lemma \ref{tryout}, let $w'\in D_k(G/Z)$ be a reflexive point, chosen such that $U_{G/Z}$ is polynomial of
Krull-dimension $n-1$ and set $\tilde w$ as in Lemma \ref{tryout}. Then $U_G=k[\tilde w^G]=
k[\tilde y_0,U_{G/Z}]$, where $\tilde y_0$ is algebraically independent of $U_{G/Z}$.
Hence $U_G$ is a polynomial ring of
Krull dimension $1+n-1=n.$
\end{proof}

The next result shows that, at least for $k=\mathbb{F}_p$, the subalgebras of Theorem \ref{arb_p_grp}
have the minimal possible Krull-dimension for polynomial retracts of $D_k$:

\begin{prp}\label{minim_standard}
Let $k=\mathbb{F}_p$, $|G|=p^n$ and $A=k[a_1,\cdots,a_m]\le D_k$ be a trace-surjective  $G$ subalgebra, then $m\ge n$. In particular, every polynomial standard subalgebra of $D_k$ has
Krull-dimension $\ge n$.
\end{prp}
\begin{proof}
Let $\Delta_k:=\prod_{g\in G}\ kb_g\cong k^{|G|}$, isomorphic to the
regular $kG$-module but with diagonal multiplication (i.e. a ``totally disconnected"
algebra). Then $\Delta_k$ is a trace-surjective $G$-algebra and there is a (necessarily surjective)
morphism $\phi:\ D_k \to \Delta_k$. Since $\phi(A)$ is a free $kG$-module, $\phi(A)=\Delta_k$.
Let $J:=\ker(\phi_{|A})$, then
$(a_i^p-a_i\ |\ i=1,\cdots,m)A\le J$ and
$p^{p^n}=\#\Delta_k=$
$$\#(A/J)\le \#k[X_1,\cdots,X_m]/(X_i^p-X_i\ |\ i=1,\cdots,m)=\#(k^{\oplus p})^{\otimes m}=p^{p^m}.$$
\end{proof}

\section{Polynomial rings of invariants and a structure theorem}\label{poly_strct}

Theorem \ref{arb_p_grp} together with the following result and Remark \ref{UG_is_poly_rem}(iii) form our First Main Theorem \ref{arb_p_grp_intro}, stated in the Introduction:

\begin{thm}\label{UG_is_poly}
Let $G$ be a finite group of order $p^n$, then there exists a trace-surjective standard
polynomial subalgebra $U_G\le D_k(G)$ of Krull-dimension $n$,
such that the ring of invariants $U_G^G$ is again a polynomial ring of Krull-dimension $n$.
\end{thm}
\begin{proof}
We use the notation of Lemma \ref{tryout} and its proof. The proof is by induction on $n$, based on Proposition \ref{cycl_p_univ_model}, for $n=1$.\\
For the induction step, assume $n>1$ and the result to be true for groups of order $p^{n-1}$.
By the induction hypothesis we have standard subalgebra, which is a polynomial algebra,
$$U_{G/Z}\cong k[\tilde y_1,\cdots,\tilde y_{n-1}]=k[w'^{G/Z}]\le D_k(G/Z),$$
with reflexive point $w'\in D_k(G/Z)\subseteq D_k(G)^Z$, such that
$U_{G/Z}^{G/Z}=k[\sigma_1,\cdots,\sigma_{n-1}]$ is a polynomial ring. By Lemma \ref{tryout} we
can extend this to a standard subalgebra,
$$U_G=k[\tilde y_0,\tilde y_1,\cdots,\tilde y_{n-1}]=k[\tilde w^G],$$
which is a polynomial algebra (see Theorem \ref{arb_p_grp})
with reflexive point $\tilde w=-\tilde y_0^{p-1}w'\in D_k(G)$ and $G$ action given by
equation (\ref{g_tilde_y0}). In particular $\tilde y_0g_0=\tilde y_0-1$; recall also that
$rh=g_0^{e_{r,h}}r_{r,h}$ $\forall r\in \mathcal{R}$ and $h\in G$,
with $e_{r,g_0}=1$, $r_{r,g_0}=r$ and $G=\cup_{r\in \mathcal{R}} \langle g_0\rangle r$.
In particular $rhg_0=g_0^{e_{r,h}+1}r_{r,h}$, hence $e_{r,hg_0}=e_{r,h}+1$.
Let $\tilde y_0g=\tilde y_0-\alpha(g)$, with $\alpha(g):=\sum_{r\in \mathcal{R}} e_{r,g}w'rg\in U_{G/Z}\le
U_G$, then
$$\beta(g):=\alpha(g)-\alpha(g)^p=\sum_{r\in \mathcal{R}} e_{r,g}(w'-w'^p)rg,$$
since $e_{r,g}\in \mathbb{F}_p$. Hence
$$\beta(gg_0)-\beta(g)=\sum_{r\in \mathcal{R}} (w'-w'^p)rg=\tr^{(G/Z)}(w'-w'^p)=1-1=0.$$
Therefore $\beta(\overline g):=\beta(g)\in U_{G/Z}$ is well-defined for $\overline g=gZ\in G/Z$.
Now $\tilde y_0^pg=\tilde y_0^p-\alpha(g)^p$, so
$(\tilde y_0-\tilde y_0^p)g=(\tilde y_0-\tilde y_0^p)-\beta(g)$, hence
$\beta(g_1g_2)=\beta(g_1){g_2}+\beta(g_2)$ for all $g_1,g_2\in G$ and therefore
$$\beta(\overline{g_1}\overline{g_2})=\beta(\overline{g_1}){\overline{g_2}}+\beta(\overline{g_2}),\
\forall \overline{g_i}\in G/Z$$
with $\beta(\overline{g})\in U_{G/Z}$. So $\beta:\ G/Z\to U_{G/Z}$ is a $1$-cocycle.
Since $U_{G/Z}$ is a free $U_{G/Z}^{G/Z}[G/Z]$-module of rank one (see Theorem \ref{first_main}),
$\beta$ is a co-boundary, hence
there is $\gamma\in U_{G/Z}$ with
\begin{equation}\label{cohomology_equation}
\gamma{\overline g}-\gamma=\beta(\overline{g})\ \forall \overline{g}\in G/Z,
\end{equation}
and so $\gamma g-\gamma=\beta(g)$ for all $g\in G$.
It follows that
$(\tilde y_0-\tilde y_0^p)g=$ $(\tilde y_0-\tilde y_0^p)-$ $(\gamma g-\gamma)$ for all $g\in G$,
therefore $\sigma_0:=\tilde y_0-\tilde y_0^p+\gamma\in U_G^G$ with $\gamma\in U_{G/Z}$.
Thus
$$U_G=k[\tilde y_0,\tilde y_1,\cdots,\tilde y_{n-1}]\ge U_G^G\ge T\le S$$
with
$T:=k[\sigma_0,\sigma_1,\cdots,\sigma_{n-1}]$ and
$S:=k[\sigma_0,\tilde y_1,\cdots,\tilde y_{n-1}]=k[\tilde y_0-\tilde y_0^p,\tilde y_1,\cdots,\tilde y_{n-1}]$,
since $\gamma\in U_{G/Z}$.
By the induction hypothesis, the ring $S$ is a free module of rank $|G/Z|=p^{n-1}$ over the polynomial ring $T$
and $U_G$ is visibly free of rank $p$ over $S$, hence
$U_G$ is a free module of rank $p^n$ over $T$. On the other hand we know from Theorem \ref{first_main},
that $U_G$ is a free module of rank $|G|=p^n$ over $U_G^G$, hence it follows easily
that $U_G^G=T$, as required.
\end{proof}

\begin{rem}\label{UG_is_poly_rem}
\begin{enumerate}
\item The proof of Theorem \ref{UG_is_poly} shows that
$U_G=k[\tilde y_0,\tilde y_1,\cdots,\tilde y_{n-1}]$ is free of rank $p$ over
$U_G^Z=k[\tilde y_0-\tilde y_0^p,\tilde y_1,\cdots,\tilde y_{n-1}]$, which itself
is free of rank $p^{n-1}$ over $U_G^G=(U_G^Z)^{G/Z}=k[\sigma_0,\sigma_1,\cdots,\sigma_{n-1}]$.
Thus $U_G$ is free over $U_G^G$ with basis consisting of monomials
$\tilde y_0^{i_0} \tilde y_1^{i_1}\cdots \tilde y_{n-1}^{i_{n-1}}$,
$0\le i_0, i_1,\cdots, i_{n-1}<p$.
\item The generating invariants $\sigma_i$ can be determined recursively, along
an upper central series $1<Z:=Z_1<Z_2<\cdots<Z_n=G$ with $Z_{i+1}/Z_i\le Z(G/Z_i)$,
by solving the cohomology-equations (\ref{cohomology_equation}) and using the formulae
\begin{equation}\label{gens_of_U}
\sigma_i=\tilde y_i-\tilde y_i^p+\gamma_i(\tilde y_{i+1},\cdots,\tilde y_{n-1})\ {\rm for}\ i=0,1,\cdots,n-1.
\end{equation}
\item It follows from equation (\ref{g_tilde_y0}) and an obvious induction argument, that $U_G$
can be constructed in such a way that $G$ acts in a ``triangular" way of the form
$(\tilde y_i)g=\tilde y_i+f_i(\tilde y_{i+1},\cdots,\tilde y_{n-1})$ for $i=0,\cdots,n-1$.
\end{enumerate}
\end{rem}

As an application of Theorem \ref{UG_is_poly}, we obtain the Structure Theorem
\ref{strct_thm_intro}. We need two little lemmas:

\begin{lemma}\label{struct_thm_lemma1}
Let $R,T\le \Omega$ and $S\le T\cap R$ be commutative rings and $I\unlhd R$ an ideal.
Then $$R/I\otimes_S T\cong R[T]/(I).$$
\end{lemma}
\begin{proof}
Both sides form the coproduct $R/I\coprod_S T$.
\end{proof}

\begin{lemma}\label{struct_thm_lemma2}
Let $R$ be a ring with modules $M, N$ and ideal
$I:=({\rm ann}_RM,{\rm ann}_RN)\unlhd R$.
Set $\overline R:= R/I$,
$\overline M:=M/{\rm ann}_RN\cdot M$ and $\overline N:=N/{\rm ann}_RM\cdot N$, then
there is an isomorphism of $R$- or $\overline R$-modules
$$M\otimes_R N\cong \overline M \otimes_{\overline R} \overline N.$$
In particular, if $\theta:\ R\to A$, $\phi:\ R\to B$ are
ring homomorphisms, then for $\overline A:=A/\theta(\ker(\phi))A$,
$\overline B:=B/\phi(\ker(\theta))B$ and $\overline R:= R/(\ker(\theta),\ker(\phi))$, there is a
canonical isomorphism of $R$- or $\overline R$-algebras
$$A\otimes_R B\cong \overline A\otimes_{\overline R} \overline B.$$
\end{lemma}
\begin{proof}
Consider the map
$$\Psi:\ \overline M\times \overline N\to M\otimes_R N,
(\bar m,\bar n)\mapsto m\otimes_R n.$$
If $\bar n=\bar n'$, then
$n-n'=\sum_{r_i} r_in_i$ with $r_i\in {\rm ann}_RM$, so
$m\otimes_R n-m\otimes_R n'=$
$$\sum_{r_i} m\otimes_Rr_in_i=\sum_{r_i} mr_i\otimes_Rn_i=0.$$
Similarly $\bar m=\bar m'$ implies $m\otimes_R n=m'\otimes_R n$,
so the map $\Psi$ is well-defined and clearly $R$-balanced.
It therefore induces a homomorphism
$$\overline M\otimes_R \overline N=
\overline M\otimes_{\overline R} \overline N\to M\otimes_R N,$$
which is easily seen to be a two sided inverse of the obvious
homomorphism
$$M\otimes_R N\to \overline M\otimes_{\overline R} \overline N.$$
\end{proof}

We now prove the second of our main theorems mentioned in the introduction, using
the notation thereof:
\\
{\bf Proof of Theorem \ref{strct_thm_intro}}:\
Since $U_G$ is a standard subalgebra, it is ``universal" (see the comment before Proposition
\ref{D_k_gen_1}). Hence there is a $G$-epimorphism
$\phi:\ U_G\to B$ onto a trace-surjective subalgebra $B\le A$, hence $B\cong U_G/\mathfrak{i}$
with $B^G=\phi(U_G^G)=k[f_0,\cdots,f_{n-1}]$ and $f_i:=\phi(\sigma_i)\in R$, where
$U_G^G=k[\sigma_0,\cdots,\sigma_{n-1}]$.
Since by Theorem \ref{first_main}, $\ker(\phi)=$ $\mathfrak{i}=$ $\mathfrak{i}^GU_G=$ $(\ker(\phi_{|U_G^G}))U_G$, Lemma \ref{struct_thm_lemma2} gives
$$A\cong R\otimes_{\phi(U_G^G)}\phi(U_G)\cong R\otimes_{U_G^G}\phi(U_G)=
R\otimes_{U_G^G}U_G=$$
$$R\otimes_{k[\sigma_0,\cdots,\sigma_{n-1}]}k[\underline Y]\cong
R[\sigma_0,\cdots,\sigma_{n-1}]/(\sigma_0-f_0,\cdots,\sigma_{n-1}-f_{n-1})\otimes_{k[\sigma_0,\cdots,\sigma_{n-1}]}
k[\underline Y],$$
with all isomorphisms being $G$-equivariant. By Lemma \ref{struct_thm_lemma1}, this latter tensor product is isomorphic to
$$R[Y_0,\cdots,Y_{n-1}]/(\sigma_0(\underline Y)-f_0,\cdots,\sigma_{n-1}(\underline Y)-f_{n-1}).$$
$\Box$

We now prove the third main theorem of the introduction:
\\
{\bf Proof of Theorem \ref{D_k_G is poly_intro}}:\
Let $G$ be a group of order $p^n$. The proof is by induction on $n$, where the
induction base $n=1$ is provided by Corollary \ref{D_k_poly_for_Cp}.
We use the notation of Lemma \ref{tryout} and Theorem \ref{UG_is_poly} and their proofs.
For $I:=(i_0,\cdots,i_{n-1}), J:=(j_0,\cdots,j_{n-1})$, we write $I\le J$ if all entries of
$J-I$ are non-negative and we write $\underline{m}$ for $(m,m,\cdots,m)$.
For $I$ and $I':=(0,i_1,...,i_{n-1})$ we define $\tilde y^I$ and $\tilde y^{I'}$
in the obvious way. We will also abuse notation slightly by writing
$(i_0,I')$ for $I$.\\
Let $B:=D_k(G)^G$ and $C:=D_k(G/Z)^{G/Z}\le B$ (see equation (\ref{D_k_G_Z})); then Proposition \ref{D_k_gen_1} yields
that $D_k(G)\cong B\otimes_{U_G^G} U_G$ and from Remark \ref{UG_is_poly_rem} (i) we see that
$$D_k(G)=\oplus_{I\le \underline{p-1}} B \tilde y^I\ {\rm and}\ D_k(G/Z)=\oplus_{I'\le \underline{p-1}} C \tilde y^{I'}.$$
Thus we can decompose $x_e\in D_k(G)$ as
$$x_e=\sum_{I\le\underline{p-1}} b_I\tilde y^I,\ b_I\in B.$$
Set $\mathcal{X}:=k[\sigma_0,C, b_I\ |\ I\le \underline{p-1}]\subseteq B$ and
$S:=\oplus_{I\le \underline{p-1}} \mathcal{X}\tilde y^I$. It follows from Remark \ref{UG_is_poly_rem} (ii)
that $\sigma_1,\cdots,\sigma_{n-1}\in C$, so $\sigma_0,\cdots,\sigma_{n-1}\in \mathcal{X}$ and
therefore powers of $\tilde y_i^s$ with $s\ge p$ can be reduced in $S$, showing that
$S=k[\mathcal{X},U]$ is a $G$-stable $k$-subalgebra of $D_k(G)$ containing $x_e$.
Hence $x_g\in S$ for all $g\in G$, therefore $D_k(G)=S$ and $\mathcal{X}=B$.
By the induction hypothesis, $C$ is a polynomial ring. It remains to show that an
appropriate number of $b_I$ are redundant generators of $B$.
\\
Consider $\tr^{(Z)}(\tilde y^I)=\tr^{(Z)}(\tilde y_0^{i_0})\tilde y^{I'}$;
then by Lemma \ref{cycl_p_lemma}, $\tr^{(Z)}(\tilde y_0 x_e)=$
$$\sum_{(i_0,I')\le \underline{p-1}} b_I \tr^{(Z)}(\tilde y_0^{i_0+1})\tilde y^{I'}=
-\sum_{I'\le \underline{p-1}} b_{p-2,I'}\tilde y^{I'}$$
if $p>2$ and
$\tr^{(Z)}(\tilde y_0 x_e)=$
$\sum_{I'\le \underline{p-1}} (b_{0,I'}+b_{1,I'})\tilde y^{I'}$
if $p=2$. On the other hand
$\tr^{(Z)}(\tilde y_0 x_e)=\tilde y_0\eta_1-\tilde x$ with
$\tilde x=x_{g_0}+2x_{g_0^2}+ \cdots=$ $\sum_{\ell=0}^{p-1}\ell x_{g_0^\ell}=$
$$\sum_{I=(i_0,I')\le \underline{p-1}}b_I(\sum_{\ell=0}^{p-1}\ell (\tilde y_0^{i_0}) {g_0}^\ell)
\tilde y^{I'}.$$
We have from Lemma \ref{cycl_p_lemma},
$$\sum_{\ell=0}^{p-1} \ell (\tilde y_0^{p-1}){{g_0}^\ell}=
\sum_{\ell=0}^{p-1} \ell (\tilde y_0-\ell)^{p-1}=-\tilde y_0\ (+1),$$
where the term ``($+1$)" only appear if $p=2$. Similarly
$$\sum_{\ell=0}^{p-1} \ell (\tilde y_0^{p-2}) {{g_0}^\ell}=
\sum_{\ell=0}^{p-1} \ell\ (\tilde y_0-\ell)^{p-2}=1$$ and
$\sum_{\ell=0}^{p-1} \ell (\tilde y_0^{i_0}) {{g_0}^\ell}=0$ for
$0\le i_0<p-2$. Therefore
$\tilde x=\sum_{I'}\big(b_{p-2,I'}-b_{p-1,I'}\tilde y_0\big)\tilde y^{I'}$ if
$p>2$ and $\tilde x=\sum_{I'}\big(b_{p-2,I'}-b_{p-1,I'}(\tilde y_0 +1)\big)\tilde y^{I'}$
if $p=2$.
Since $\eta_1=\sum_{I'}c_{I'}\tilde y^{I'}$ with $c_{I'}\in C$, we
get for $p>2$:
$$\tr^{(Z)}(\tilde y_0 x_e)=-\sum_{I'\le \underline{p-1}} b_{p-2,I'}\tilde y^{I'}=
\sum_{I'\le \underline{p-1}}
\big[\big(-b_{p-2,I'} \big)\tilde y^{I'} +
(c_{I'}+b_{p-1,I'})\tilde y_0\tilde y^{I'}\big],$$
while for $p=2$ we get:
$$\tr^{(Z)}(\tilde y_0 x_e)=\sum_{I'\le \underline{p-1}} (b_{0,I'}+b_{1,I'})\tilde y^{I'}=
\sum_{I'\le \underline{p-1}}
\big[\big(b_{0,I'}+b_{1,I'}\big)\tilde y^{I'} +
(c_{I'}+b_{p-1,I'})\tilde y_0\tilde y^{I'}\big].$$
Hence $b_{p-1,I'}=-c_{I'}\in C$. Moreover, since
$t_{\mathcal{R}}(\tilde y^{I'})=(-1)^{n-1}$ for\\
$I'=(0,p-1,\cdots,p-1)$ and zero for $I'\le (0,p-1,\cdots,p-1)$
and $I'\ne (0,p-1,\cdots,p-1)$, we get
$t:={\rm tr}(\tilde y_0 x_e)=$
$$-\sum_{I'\le \underline{p-1}} b_{p-2,I'}t_{\mathcal{R}}(\tilde y^{I'})=
(-1)^nb_{p-2,\underline{p-1}},$$
if $p>2$ and $t=b_{0,\underline{p-1}}+b_{1,\underline{p-1}}$ if $p=2$.
On the other hand, by Lemma \ref{tryout}, $\tilde y_0 g\in
k[\tilde y_0, D_k(G/Z)]=k[y_0, D_k(G/Z)]$
and $\eta_1\in D_k(G/Z)$, hence
$t=t_{\mathcal{R}}(\tilde y_0 \eta_1)-t_{\mathcal{R}}(\tilde x)=$
$t_{\mathcal{R}}(\tilde y_0 \eta_1)-y_0\in k[\tilde y_0, D_k(G/Z)]$
(recall that $y_0=\tilde y_0-\mu$ for some $\mu\in D_k(G/Z)$.)
Using the relation $$\tilde y_0^p=\tilde y_0+\gamma_0(\tilde y_1,\cdots,\tilde y_{n-1})-\sigma_0$$
we get
$$t=\sum_{I\le \underline{p-1}} \epsilon_I \tilde y^I\in
(\oplus_{I\le \underline{p-1}} C[\sigma_0]\tilde y^I)\cap B=C[\sigma_0].$$
If $p>2$ we conclude that $t=(-1)^n b_{p-2,\underline{p-1}}\in C[\sigma_0]$
and if $p=2$ we get $b_{0,\underline{p-1}}+b_{1,\underline{p-1}}\in C[\sigma_0]$,
so $b_{0,\underline{p-1}}\in C[\sigma_0]$ as well.
It follows that
$$B=\mathcal{X}=k[\sigma_0,C,\{b_{i_0,I'}\ |\ 0\le i_0< p-1\}\backslash \{b_{p-2,\underline{p-1}}\}].$$
Since $C$ is generated by $p^{n-1}-1$ elements, the number of
generators for $B$ is
$$1+ (p^{n-1}-1)+(p^n-p^{n-1}-1)=p^n-1.$$
Hence the result follows, since $B=D_k(G)^G$ has Krull-dimension $p^n-1$.
$\Box$

\section{Concluding remarks}\label{concl_rem}

In the light of the Structure Theorem \ref{strct_thm_intro},
it is interesting and also challenging to construct
explicit standard polynomial algebras $U\le D_k$ of Krull-dimension
$\mathrm{log}_p(|G|)$ and their rings of invariants, for particular
classes of $p$-groups. More explicit versions of the Structure Theorem
\ref{strct_thm_intro} will depend on the finer structure of the $p$-groups
considered. The authors have started this investigation for example for
cyclic groups of order larger than $p$, general abelian $p$-groups and
and extraspecial $p$-groups. The results will appear in subsequent papers.
\\
In the context of Theorem \ref{arb_p_grp} it would be interesting to classify all standard
subalgebras of $D_k$ and particularly to check whether
they exist in Krull-dimension $<\mathrm{log}_p(|G|)$. In the case $k=\mathbb{F}_p$, such a
standard subalgebra cannot be polynomial, due to Proposition \ref{minim_standard}.
Therefore, if it exists, it would provide a counterexample to the conjecture discussed
in \cite{costa}, that all retracts of polynomial rings are again polynomial.\\
In this context we are also interested in the structure of the category $\mathfrak{Ts}$ for
a general finite $p$-group. Here are some question for which we only have partial answers:

\begin{itemize}
\item[-] We know that $A\in \mathfrak{Ts}$ is a finitely generated projective object if and only
if $A\oplus J=D_k^{\otimes m}$ for some $m$ with \emph{ideal} $J$,
i.e. $A$ is a \emph{retract}  of $D_k^{\otimes m}$.
\item[$\bullet$]  We don't know if all of these are polynomial rings. This is related to
the aforementioned general problem in algebra, whether every retract of a
polynomial ring is again a polynomial ring (see \cite{costa}, \cite{shpilrain}).
\item[-] We know that $U_k$ is a projective object in $\mathfrak{Ts}$; moreover if
$k=\mathbb{F}_p$, then $n=\mathrm{log}_p(|G|)$ is the minimal possible
Krull-dimension for such an object being a polynomial ring.
\item[$\bullet$] We don't know what the projective objects of minimal dimension
in $\mathfrak{Ts}$ are in general.
\end{itemize}
A further line of research opens up, if the finite $p$-group $G$ is replaced
by an arbitrary finite group. As mentioned before, some of our results
have variants in that general context, but the theory needs further development.
It should also be mentioned that some aspects of the presented theory make perfect
sense and promise interesting results in invariant theory and representation theory,
if the commutative $k$-algebra $A$ is replaced
by a non-commutative $k-G$-algebra with surjective trace.

\bibliography{bibl}
\bibliographystyle{plain}

\affiliationone{
   P. Fleischmann and C.F. Woodcock\\
   School of Mathematics, Statistics\\
   and Actuarial Science\\
   University of Kent\\
   Canterbury CT2 7NF\\
   United Kingdom
   \email{P.Fleischmann@kent.ac.uk\\
   C.F.Wooodcock@kent.ac.uk}}
\end{document}